\documentclass[reqno]{amsart}
\usepackage[utf8]{inputenc}
\usepackage{cite}
\usepackage{xcolor}
\usepackage{color}
\usepackage{calligra}
\usepackage[T1]{fontenc}
\usepackage{amsthm}
\usepackage{latexsym}
\usepackage[all]{xy}
\usepackage{hyperref}
\usepackage{mathrsfs}
\usepackage{amsmath,amsfonts,amssymb,amsxtra}
\usepackage{graphicx}
\usepackage{amscd}
\numberwithin{equation}{section}
\newtheorem{theorem}{\bf Theorem}[section]
\newtheorem{lem}{\bf Lemma}[section]
\newtheorem{cor}{\bf Corollary}[section]
\newtheorem{remark}{\bf Remark}[section]

\newcommand\dv{\mathrm{div}}
\newcommand\tr{\mathrm{tr}}
\newcommand\rc{\mathrm{Ric}}
\newcommand\dm{\mathrm{diam}}
\begin{document}

\title[Estimates for the first eigenvalues of Bi-drifted Laplacian]{Estimates for the first eigenvalues of Bi-drifted Laplacian on smooth metric measure space}
\author[Marcio C. Araújo Filho]{Marcio Costa Araújo Filho$^1$}
\address{$^1$Departamento de Matemática, Universidade Federal de Rondônia, Campus Ji-Paraná, R. Rio Amazonas, 351, Jardim dos Migrantes, 76900-726 Ji-Paraná, Rondônia, Brazil}
\email{$^1$marcio.araujo@unir.br}

\begin{abstract}
In this paper we obtain lower bounds for the first eigenvalue to some kinds of the eigenvalue problems for Bi-drifted Laplacian operator on compact manifolds (also called a smooth metric measure space) with boundary and $m$-Bakry-Emery Ricci curvature or Bakry-Emery Ricci curvature bounded below. We also address the eigenvalue problem with Wentzell-type boundary condition for drifted Laplacian on smooth metric measure space.
\end{abstract}
\maketitle

\section{Introduction}
Let $(M, \langle , \rangle)$ be an $n$-dimensional compact Riemannian manifold with boundary. For a given  function $\phi \in C^2(M)$ the triple $(M, \langle , \rangle, e^{-\phi}dv)$ is customarily called a smooth metric measure space or manifolds with density, where $dv$ is the Riemannian volume measure on $M$. On $(M, \langle , \rangle, e^{-\phi}dv)$ we defined the drifted Laplacian operator (also called $\phi$-Laplacian or Witten– Laplacian) as follows 
\begin{equation*}
     \Delta_\phi = \Delta  -  \langle \nabla \phi,\nabla \rangle,
\end{equation*} 
where $\Delta$ and $\nabla$ are the Laplacian and the gradient operator on $M$, respectively. Moreover, let us consider also the  Bi-drifted Laplacian $ \Delta_\phi^2 =  \Delta_\phi ( \Delta_\phi )$.

In this paper, we are interested in obtaining estimates for the first eigenvalue of problems for the Bi-drifted Laplacian operator on smooth metric measure space with $m$-Bakry-Emery Ricci curvature or Bakry-Emery-Ricci curvature bounded below. In this direction, interesting estimates have been obtained recently in \cite{DuBezerra}, \cite{HuangMa}, \cite{LiWei}, \cite{WeiWylie}, \cite{TuHuang} and others. We also address an eigenvalue problem with the Wentzell-type boundary condition for drifted Laplacian on smooth metric measure space, see Theorem~\eqref{wentzell-theorem} and Theorem~\eqref{steklov-theorem}.  

On $(M, \langle , \rangle, e^{-\phi}dv)$, the $m$-dimensional Bakry-Emery Ricci curvature or $m$-Bakry-Emery Ricci curvature is defined by
\begin{equation*}
\rc_\phi^m = \rc + \nabla^2\phi - \frac{1}{m-n}\nabla\phi \otimes \nabla\phi,
\end{equation*}
where $m\geq n$ is a constant, $\rc$ is the Ricci curvature on $M$, $\nabla^2$ be the Hessian operator and $m=n$ if and only if $\phi$ is a constant (see  \cite{BakryEmery}, \cite{HuangMa} and \cite{LiWei}). Let us define 
\begin{equation*}
    \rc_\phi = \rc + \nabla^2 \phi,
\end{equation*}
so that $\rc_\phi$ can be seen as the $\infty$-dimensional Bakry-Emery Ricci curvature also called Bakry-Emery Ricci curvature. It is worth mentioning that $\rc_\phi^m$  plays as a interesting substitute of the Ricci curvature for establishing many important results in differential geometry, see \cite{BezerraXia}, \cite{DuBezerra}, \cite{HuangMa}, \cite{TuHuang} and references therein. For example, using $\rc_\phi^m$, Qian~\cite{Qian} obtained a generalization of the well-known Myer's theorem. In fact, it showed that if $\rc_\phi^m\geq (m-1)c>0$ then M should be compact with a limited diameter, that is, $\dm (M)\leq \frac{\pi}{\sqrt{c}}$.

In recent years, the Ricci Flow has been the interest of many mathematicians and the equation $\rc_\phi = \lambda\langle , \rangle$ for some constant $\lambda$ appears as the gradient Ricci soliton equation playing a key role in this study(see \cite{Cao}). The gradient Ricci soliton is classified according to the  sign of $\lambda$, that is, $(M, \langle , \rangle,e^{-\phi}dv, \lambda)$ is called steady for $\lambda = 0$, shrinking for $\lambda > 0$ and expanding for $\lambda<0$.
Furthermore, when $m>n$, the equation $\rc_\phi^m = \kappa\langle , \rangle$ for some constant $\kappa$ is a generalized $\overline{m}$-quasi-Einstein metric equation, where $\overline{m}=m-n>0$ (see, for example, \cite[Definition~1]{BarrosRibeiro} and \cite{BarrosGomes}). An interesting motivation for the $m$-Bakry-Emery Ricci curvature was given by Wei and Wylie~\cite[Section~2]{WeiWylie}. 

Our initial results concern the problems of Blucking and clampled plate for Bi-drifted Laplacian on smooth metric measure space $(M, \langle , \rangle, e^{-\phi}dv)$, that is,
\begin{equation}\label{clampledproblem}
    \left \{ \begin{array}{ll}\Delta_\phi^2 u =\Gamma u & \mbox{in}\quad M,  \\
    u = \frac{\partial u}{\partial {\nu}}= 0 & \mbox{on}\quad \partial M, \end{array} \right.
\end{equation}
\begin{equation}\label{buckling}
    \left \{ \begin{array}{ll} \Delta_\phi^2 u = -\Lambda \Delta_\phi u & \mbox{in} \quad M,  \\
    u =\frac{\partial u}{\partial \nu} = 0 & \mbox{on} \quad \partial M. \end{array} \right.
\end{equation}
When $\phi$ is a constant, many important results  have been obtained for Problems~\eqref{clampledproblem} and \eqref{buckling}, see \cite{ashb2}, \cite{ashb3}, \cite{ashb4}, \cite{ccwx}, etc. When $\phi$ is not necessary constant, the eigenvalues of problems \eqref{clampledproblem} and \eqref{buckling} have been studied in many papers, see  \cite{BezerraXia}, \cite{DuBezerra} and references therein. In fact, we get the following two results.
\begin{theorem}\label{clampledteo}
Let $(M, \langle , \rangle, e^{-\phi}dv)$ be an $n(\geq 2)$-dimensional compact connected smooth measure space with boundary $\partial M$ and denote by $\nu$ the outward unit normal vector field of $\partial M$. Let $\lambda_1$ be the first Dirichlet eigenvalue of the drifted Laplacian of $M$ and let $\Gamma_1$ be the first eigenvalue of Problem~\eqref{clampledproblem}. 
\begin{enumerate}
\item \label{thm1-1} Assume that  the Bakry-Emery Ricci curvature of $M$ is bounded below by $\frac{|\nabla \phi|^2}{na}+b$, 
for some positive constants $a$ and $b$. Then,
\begin{align*}
    \Gamma_1 > \lambda_1\Big(\frac{\lambda_1}{n(a+1)} + b \Big).
\end{align*}
\item \label{thm1-2} Assume that the $m$-Bakry-Emery Ricci curvature of $M$ is bounded below by $(m-1)k \geq 0$. Then,
\begin{align*}
    \Gamma_1 > \lambda_1\Big(\frac{\lambda_1}{m}+(m-1)k\Big).
\end{align*}
\end{enumerate}
\end{theorem}

\begin{theorem}\label{bucklingteo}
Let $(M, \langle , \rangle, e^{-\phi}dv)$ be an $n(\geq 2)$-dimensional compact connected smooth measure space with boundary $\partial M$ and denote by $\nu$ the outward unit normal vector field of $\partial M$. Assume that the Bakry-Emery Ricci curvature of $M$ is bounded below by $\frac{|\nabla \phi|^2}{na}+b$, for some positive constants $a$ and $b$. Let $\lambda_1$ be the first Dirichlet eigenvalue of the drifted Laplacian of $M$. Then, the first eigenvalue of Problem~\eqref{buckling} satisfies
\begin{equation*}
    \Lambda_1 > \frac{\lambda_1}{n(a+1)} + b. 
\end{equation*}
\end{theorem}
Similar results as in the previous theorems were made by Bezerra and Xia in \cite{BezerraXia} changing the boundary conditions.

Steklov in \cite{Stekloff} studied elliptic problems with parameters in the boundary conditions. From then on, these types of problems came to be called Steklov problems.
In this paper, we are also interested in the following Steklov problems on smooth metric measure space $(M, \langle , \rangle, e^{-\phi}dv)$:
\begin{equation}\label{volumproblem}
    \left \{ \begin{array}{ll}\Delta_\phi^2 u = 0 & \mbox{in} \quad M,  \\
    u = \overline{\Delta}_\phi u - p\frac{\partial u}{\partial \nu} = 0 & \mbox{on} \quad \partial M, \end{array} \right.
\end{equation}
\begin{equation}\label{problemxi1}
    \left \{ \begin{array}{ll} \Delta_\phi^2 u = 0 & \mbox{in} \quad M,  \\
    \frac{\partial u}{\partial \nu} = \frac{\partial (\Delta_\phi u)}{\partial \nu} + \xi u = 0 & \mbox{on} \quad \partial M, \end{array} \right.
\end{equation}
\begin{equation}\label{problemzeta1}
    \left \{ \begin{array}{ll} \Delta_\phi^2 u = 0 & \mbox{in} \quad M,  \\
    \frac{\partial u}{\partial \nu} = \frac{\partial (\Delta_\phi u)}{\partial \nu} + \beta \overline{\Delta}_\phi u + \zeta u = 0 & \mbox{on} \quad \partial M, \end{array} \right.
\end{equation}
and
\begin{equation}\label{problemtau1}
    \left \{ \begin{array}{ll}  \Delta_\phi^2 u = 0 & \mbox{in} \quad M,  \\
    \frac{\partial u}{\partial \nu} = \frac{\partial (\Delta_\phi u)}{\partial \nu} - \tau \overline{\Delta}_\phi u = 0 & \mbox{on} \quad \partial M, \end{array} \right.
\end{equation}
where $\overline{\Delta}_\phi$ is the drifted Laplacian on $\partial M$.

When $\phi=\mbox{constant}$ the Bi-drifted Laplacian becomes  the biharmonic operator and, in this case, Chen et al.  \cite{ccwx} obtained lower limits for the first eigenvalue of some kinds of eigenvalues problems on compact manifolds with boundary positive Ricci curvature. When $\phi$ is not necessarily constant, that is, for the Bi-drifted Laplacian case, Du and Bezerra~\cite{DuBezerra} obtained the results of Chen et al.~\cite{ccwx}. 
More recently Wang-Xia~\cite{wangxia4} achieved results similar to those found by Chen et al.~\cite{ccwx}, for the first eigenvalue of four kinds of eigenvalue problems of the biharmonic operator. In this paper, we generalize some  results found by Wang-Xia~\cite{wangxia4}, for the Bi-drifted Laplacian in an smooth metric measure space $(M, \langle , \rangle, e^{-\phi}dv)$ with $m$-Bakry-Emery Ricci curvature or Bakry-Emery Ricci curvature bounded inferiorly. 
 
Problem~\eqref{problemxi1} was first studied in \cite{kttlr3} by Kuttler and Sigillito when $\Delta_\phi = \Delta$ ($\phi=\mbox{constant}$) and $M=\Omega$ is a bounded domain in $\mathbb{R}^n$, in this case they gave some estimates for the first nonzero eigenvalue $\xi_1$. Also, Xia and Wang in \cite{wangxia3} proved an isoperimetric upper bound for $\xi_1$ when $M$ is a bounded domain in $\mathbb{R}^n$ and $\Delta_\phi = \Delta$. We can characterize $\xi_1$ using the Rayleigh-Ritz formula, that is
\begin{align}\label{eq-xi1}
    \xi_1=\min_{\substack{0\neq w \in H^2(M) \\
\int_{\partial M}wd\vartheta=0=w|_{\partial M}}}\frac{\int_M (\Delta_\phi w)^2d\mu}{\int_{\partial M}w^2d\vartheta},
\end{align}
where $d\mu=e^{-\phi}dv$ and  $d\vartheta=e^{-\phi}dA$ are the Riemannian volume measure on $M$ and $\partial M$, respectively, which will be the notation used throughout the paper.

The first nonzero eigenvalue of Problem~\eqref{problemzeta1} can be characterized as 
\begin{equation*}
    \zeta_{1, \beta}=\min_{\substack{0\neq w \in H^2(M) \\
\int_{\partial M}wd\vartheta=0=w|_{\partial M}}}\frac{\int_M (\Delta_\phi w)^2d\mu+\beta\int_{\partial M}|\overline{\nabla}w|^2d\vartheta}{\int_{\partial M}w^2d\vartheta}.
\end{equation*}

Let $\eta_1$ the first nonzero eigenvalue of the drifted Laplacian of $\partial M$. If $\beta >0$ and $\xi_1$ is the first nonzero eigenvalue of the Steklov problem \eqref{problemxi1} we have the following relationship
\begin{equation}\label{eq-zeta1-xi1}
    \zeta_{1, \beta} \geq \xi_1 + \beta \eta_1,
\end{equation}
with equality holding if and only if any eigenfunction $w$ corresponding to $\zeta_{1, \beta}$ is an eigenfunction corresponding to $\xi_1$ and $w|_{\partial M}$ is an eigenfunction corresponding to $\eta_{1, \beta}$. We have the following result.
\begin{theorem}\label{zeta1estimateteo}
Let $(M, \langle , \rangle, e^{-\phi}dv)$ be an $n$-dimensional compact smooth measure space with boundary $\partial M$ and the $m$-Bakry-Emery Ricci curvature of $M$ is bounded below by $-(m-1)k$ for some constant $k \geq 0$. Assume that the principal curvatures of $\partial M$ are bounded below by a positive constant $c$ and denote  by $\zeta_{1, \beta}$ be the first eigenvalue of Problem~\eqref{problemzeta1}. Then,
\begin{equation*}\label{zeta1estimate}
    \zeta_{1, \beta} > \frac{mc\eta_1\mu_1}{(m-1)(\mu_1 +mk)} + \beta\eta_1,
\end{equation*}
where $\mu_1$ and $\eta_1$ are the first nonzero Neumann eigenvalue of the drifted Laplacian of $M$ and the first nonzero eigenvalue of the drifted Laplacian of $\partial M$, respectively.
\end{theorem}

From the proof of the theorem above we have the following result.
\begin{cor}
Under the same setup as in Theorem~\ref{zeta1estimateteo}, the first nonzero eigenvalue of Problem~\eqref{problemxi1} satisfies 
\begin{equation*}
    \xi_1 > \frac{mc\eta_1\mu_1}{(m-1)(\mu_1 +mk)}.
\end{equation*}
\end{cor}
\begin{remark}
When $\phi$ is a constant, then $m=n$ and $\rc_\phi^m=\rc$. In this case, Theorem~\ref{zeta1estimateteo} implies \cite[Theorem~1.6]{wangxia4}. For $\phi$ not necessarily constant, Bezerra and Xia in \cite{BezerraXia} obtained a estimate for first eigenvalue of Problem~\eqref{problemzeta1} supposing a lower limitation on the Bakry-Emery Ricci curvature. While we are using $\rc_\phi^m \geq -(m-1)k$, for some $k>0$, Bezerra and Xia~\cite[Theorem~1.5]{BezerraXia} used $\rc_\phi \geq \frac{|\nabla \phi|^2}{na}-b$, for some positive constants $a$ and $b$.
\end{remark}

When $M=B$ is a unit ball in euclidean space and $\phi$ is constant, Wang and Xia \cite{wangxia3} explicitly described the eigenvalues and the corresponding eigenfunctions to the Problem~\eqref{problemtau1} (see \cite[Theorem~1.7]{wangxia3}). Our next result is a lower bound for the first nonzero eigenvalue of Problem~\eqref{problemtau1}.
\begin{theorem}\label{teorematau1}
Let $(M, \langle , \rangle, e^{-\phi}dv)$ be an $n$-dimensional compact smooth measure space with boundary $\partial M$. Suppose that the principal curvatures of $\partial M$ are bounded below by a positive constant $c$ and denote by $\tau_1$ the first eigenvalue of Problem~\eqref{problemtau1} and $\mu_1$ is the first nonzero Neumann eigenvalue of the drifted Laplacian of $M$. 
\begin{enumerate}
    \item \label{thm4-1} Assume that the Bakry-Emery Ricci curvature of $M$ is bounded below by  $ \frac{|\nabla\phi|^2}{na}-b$, for some positive constants $a$ and $b$. Then, 
\begin{equation*}
    \tau_1> \frac{nc(a+1)\mu_1}{n(a+1)(\mu_1+b)-\mu_1}.
\end{equation*}
\item \label{thm4-2} Assume that the $m$-Bakry-Emery Ricci curvature of $M$ is bounded below by $-(m-1)k$ for some constant $k\geq 0$. Then, 
\begin{equation*}
    \tau_1 > \frac{mc\mu_1}{(m-1)(\mu_1+mk)}.
\end{equation*}
\end{enumerate}
\end{theorem}

With respect to Problem~\eqref{volumproblem} we get the next result.
\begin{theorem}\label{volumproblemteo}
Let $(M, \langle , \rangle, e^{-\phi}dv)$ be an $n(\geq 2)$-dimensional compact connected smooth metric measure space with boundary $\partial M$ and denote by $\nu$ the outward unit normal vector field of $M$. Then, the first nonzero eigenvalue $p_1$ of Problem~\eqref{volumproblem} satisfies $p_1 \leq \frac{A_\phi}{V_\phi}$, where $A_\phi = \int_{\partial M} e^{-\phi}dA$, $V_\phi=\int_{M}e^{-\phi}dv$. Moreover, if in addition the Bakry-Emery Ricci curvature of $M$ satisfies $\rc_\phi \geq \frac{|\nabla \phi|^2}{na}$, for some positive constant $a$, and there is a point $x_0 \in \partial M$ such that the weighted mean curvature of $\partial M$ satisfies $H_\phi(x_0) \geq \frac{n(a+1)-1}{n(a+1)(n-1)}\frac{A_\phi}{V_\phi}$, then $p_1=\frac{A_\phi}{V_\phi}$ implies that $M$ is isometric to an $n$-dimensional Euclidean ball and $\phi$ is constant.
\end{theorem}

\begin{remark}
The inequality $p_1 \leq \frac{A_\phi}{V_\phi}$ was shown by Huang and Ma in \cite{HuangMa}, our main contribution was made in the second part of the theorem. Moreover, Theorem~\eqref{volumproblemteo} is a generalization for drifted Laplacian of the \cite[Theorem~1.3]{wangxia1} by Wang and Xia. Such a generalization is in the following sense, for $\phi$ constant, we get $p_1 \leq \frac{A}{V}$ where $A= \int_{\partial M}dA$ and $V=\int_{M} dv$, which is the same result obtained in the first part of \cite[Theorem~1.3]{wangxia1}. Furthermore, for $\phi$ constant, $H_\phi$ becomes $H$ and our hypothesis in the second part of Theorem~\ref{volumproblemteo} becomes $\rc \geq 0$ and there is a point $x_0 \in \partial M$ such that $H(x_0) \geq \frac{n(a+1)-1}{n(a+1)(n-1)}\frac{A}{V}$ which implies that there is a point $x_0 \in \partial M$ such that $H(x_0) > \frac{A}{nV}$, since $a>0$, which is enough to get the second part of \cite[Theorem~1.3]{wangxia1}.
\end{remark}

Now, on a smooth metric measure space $(M^n, \langle , \rangle, e^{-\phi}dv)$, let us consider the following eigenvalue problem with the Wentzell-type boundary condition for drifted Laplacian:
\begin{equation}\label{wentzell-problem}
    \left \{ \begin{array}{ll}\Delta_\phi u = 0 & \mbox{in } M,  \\
    -\rho \overline{\Delta}_\phi u + \frac{\partial u}{\partial \nu} = \gamma u & \mbox{on } \partial M, \end{array} \right.
\end{equation}
where $\rho$ is a given real number. For $\rho \geq 0$, the spectrum of Problem~\eqref{wentzell-problem} consists in an increasing countable sequence of eigenvalues 
\begin{equation*}
    \gamma_{0, \rho}=0 < \gamma_{1, \rho} \leq \gamma_{2, \rho} \leq \cdots \to + \infty.
\end{equation*}

By the variational principle, we can see that the first nonzero eigenvalue $\gamma_{1, \phi}$ of Problem~\eqref{wentzell-problem} has the following characterization:
\begin{equation}\label{gamma1}
    \gamma_{1, \phi} = \min \Big\{\frac{\int_M|\nabla u|^2d\vartheta+\rho\int_{\partial M}|\overline{\nabla}u|^2 d\mu}{\int_{\partial M} u^2 d\mu} \Big| u\in W^{1,2}(M), u\neq0, \int_{\partial M}ud\mu=0  \Big\}. 
\end{equation}

The next result is a lower estimate for the first nonzero eigenvalue of Problem~\eqref{wentzell-problem}.
\begin{theorem}\label{wentzell-theorem}
Let $(M, \langle , \rangle, e^{-\phi}dv)$ be an $n(\geq 2)$-dimensional compact connected smooth metric measure space with boundary $\partial M$ and denote by $\nu$ the outward unit normal vector field of $M$. If $\rc_\phi^m \geq -k$ for some non-negative constant $k$, the second fundamental form of $\partial M$ satisfies $II\geq cI$ (in the quadratic form sense), and $H_\phi \geq c$ for some positive constant $c$, then the first nonzero eigenvalue of Problem~\eqref{wentzell-problem} satisfies
\begin{equation}\label{gamma1estimate}
    \gamma_{1, \phi} \leq \rho \eta_1 + \frac{2\eta_1+k+\sqrt{(2\eta_1+k)^2-4(n-1)\eta_1c^2}}{2(n-1)c},
\end{equation}
where $\eta_1$ is the first nonzero closed eigenvalue of the drifted Laplacian on the boundary $\partial M$. Equality in \eqref{gamma1estimate} holds if and only if $k=0$ and $M$ is isometric to an $n$-dimensional Euclidean ball of radius $\frac{1}{c}$, $\phi$ is the nonzero constant function and $m=n$.
\end{theorem}

We can see that when $\rho=0$ the Problem~\eqref{wentzell-problem} becomes the second order Steklov problem for drifted Laplacian:
\begin{equation}\label{steklov-problem}
    \left \{ \begin{array}{ll}\Delta_\phi u = 0 & \mbox{in } M,  \\
    \frac{\partial u}{\partial \nu} = q u & \mbox{on } \partial M, \end{array} \right.
\end{equation}
which has a discrete spectrum consisting in a sequence
\begin{equation*}
    q_0=0 < q_1 \leq q_2 \leq \cdots \to + \infty.
\end{equation*}
In this case, we can to obtain the following result.
\begin{theorem}\label{steklov-theorem}
Under the same setup as in Theorem~\ref{wentzell-theorem}, the first nonzero eigenvalue of Problem~\eqref{steklov-problem} satisfies
\begin{equation*}
    q_1 > \frac{c\eta_1}{2\eta_1+k},
\end{equation*}
where $\eta_1$ is the first nonzero closed eigenvalue of the drifted Laplacian on the boundary $\partial M$.
\end{theorem}

\begin{remark}
Theorem~\eqref{wentzell-theorem} generalize a result obtained by Xia and Wang in \cite{wangxia3} and a result obtained by Zhao et al.\cite{Zhao}. In fact, if $\phi$ is constant the Theorem~\eqref{wentzell-theorem} implies \cite[Theorem~1.1]{wangxia3} and if $k=0$ Theorem~\eqref{wentzell-theorem} implies \cite[Theorem~4.1]{Zhao}. Moreover, when $\phi$ is a constant Theorem~\eqref{steklov-theorem} implies \cite[Theorem~1.3; i)]{wangxia3} and if $k=0$ we get $q_1>\frac{c}{2}$ which is Escobar's well-known estimate \cite[Theorem~8]{Escobar}.
\end{remark}

Reilly's formula \eqref{reillyequality} and Inequality~\eqref{reillyinequality} for drifted Laplacian play an important role in our demonstrations. The method used in our demonstrations is classic and has been widely used in papers in the bibliography. For instance, our proofs are mainly motivated by the corresponding results for the Bi-Laplacian case (see \cite{ccwx}, \cite{wangxia4}) and by similar results for drifted and Bi-drifted Laplacian (see \cite{BezerraXia}, \cite{DuBezerra}, \cite{HuangMa}, \cite{Zhao}). 

\section{Proof of Theorems}
In order to prove the results of the previous section let us fix some notations and remember facts known as the famous Reilly's formula. For $n$-dimensional compact manifold  $M$ with boundary $\partial M$, let us denote $\langle , \rangle$ the Riemannian metric on $M$ as well as that induced on $\partial M$. Let us consider ${\nabla}$ and $\Delta_\phi$ the connection and the drifted Laplacian on $M$, respectively, and denote by $\nu$ be the unit outward normal vector of $\partial M$, so that the shape operator of $\partial M$ is given by $S(X) = \nabla_X \nu $ and the second fundamental form of $\partial M$ is defined as $II(X, Y) =\langle S(X), Y\rangle$, where $X, Y \in T\partial M$. We call the principal curvatures of $\partial M$ to eigenvalues of shape operator and the mean curvature $H$ of $\partial M$ is given by $H=\frac{1}{n} \tr S$, where $\tr$ is the trace operator calculated in the metric $\langle , \rangle$. 

Ma and Du~\cite{MaDu} studied the drifted Laplacian and obtained an extension to classical Reilly's formula on compact smooth metric measure space $(M, \langle , \rangle, e^{-\phi}dv)$ with boundary $\partial M$. In fact, for a $f\in C^\infty(M)$ they showed the following identity
\begin{equation}\label{reillyequality}
\begin{split}
&\int_M \Big[ ({\Delta_\phi}f)^2 - |{\nabla}^2f|^2-\rc_\phi({\nabla}f, {\nabla}f) \Big]e^{-\phi}dv\\
&=\int_{\partial M}\Big[2(\overline{\Delta}_\phi f) f_\nu +(n-1)H_\phi f_\nu^2 + II (\overline{\nabla} f, \overline{\nabla} f) \Big]e^{-\phi}dA,
\end{split}
\end{equation}
where $ H_\phi = H - \frac{1}{n-1}\phi_\nu$ denotes the weighted mean curvature (also called $\phi$-mean curvature, here $\phi_\nu=\frac{\partial \phi}{\partial \nu}$), $\nabla^2 f$ is the Hessian of $f$; $\overline{\Delta}$, $\overline{\Delta}_\phi$ and $\overline{\nabla}$  represent the Laplacian operator, the drifted Laplacian operator and the gradient on $\partial M$, with respect to the induced metric on $\partial M$, respectively. Sometimes, we will use the notations $d\mu=e^{-\phi}dv$ and $d\vartheta = e^{-\phi} dA$, where $dv$ and $dA$ are the Riemannian volume measure on $M$ and $\partial M$, respectively.

Reilly's formula \eqref{reillyequality} can be modified for an inequality where we get $\rc_\phi^m$ instead of $\rc_\phi$. First remember that Schwarz inequality is 
\begin{equation*}
    |\nabla^2f|^2\geq \frac{1}{n}(\Delta f)^2,
\end{equation*}
with equality holding if and if $\nabla^2f=\frac{\Delta f}{n}\langle, \rangle$. Consider the following basic algebraic inequality
\begin{equation}\label{algebric-ineq}
    (a+b)^2\geq \frac{a^2}{\alpha+1} - \frac{b^2}{\alpha} \quad \mbox{for all} \quad \alpha >0,
\end{equation}
with equality holding if and only if $a+\frac{\alpha+1}{\alpha}b=0$.
Substituting $\alpha=\frac{m-n}{n}$ in Inequality~\eqref{algebric-ineq} and using Schwarz inequality we obtain, when $m>n$, (see \cite{DuBezerra}, \cite{HuangMa}, \cite{TuHuang})
\begin{equation}\label{eq2.2}
    |\nabla^2f|^2 \geq \frac{1}{n}(\Delta f)^2= \frac{1}{n}(\Delta_\phi f + \langle \nabla \phi, \nabla f \rangle)^2 \geq \frac{(\Delta_\phi f)^2}{m}-\frac{\langle \nabla \phi, \nabla f \rangle}{m-n}^2.
\end{equation}
Thus, inserting \eqref{eq2.2} into \eqref{reillyequality} we get 
\begin{equation}\label{reillyinequality}
\begin{split}
&\int_M \Big[\frac{m-1}{m}({\Delta_\phi}f)^2 - \rc_\phi^m({\nabla}f, {\nabla}f) \Big]d\mu\\
&\geq \int_{\partial M}\Big[2(\overline{\Delta}_\phi f) f_\nu +(n-1)H_\phi (f_\nu)^2 + II (\overline{\nabla} f, \overline{\nabla} f) \Big]d\vartheta.
\end{split}
\end{equation}

\begin{remark}\label{remar2.1}
It is important to note that:
\begin{itemize}
    \item[i)] When $m=n$, we know that $\phi$ is a constant, then $\Delta_\phi = \Delta$ and \eqref{reillyequality} becomes classic Reilly's formula \cite{reilly}, and the equality in \eqref{reillyinequality} holds if and only if $\nabla^2f=\frac{\Delta f}{n}\langle, \rangle$.
    \item[ii)] When $m>n$, the equality holds in \eqref{reillyinequality} if and only if $\nabla^2f=\frac{\Delta f}{n}\langle, \rangle$ and $\Delta_\phi f +\frac{m}{m-n}\langle \nabla \phi, \nabla f \rangle=0$.
\end{itemize}
\end{remark}

We will need of the two following lemmas.
\begin{lem}[\cite{BezerraXia}]\label{lemma}
Let $M$ an $n$-dimensional Riemannian manifold with boundary $\partial M$ and let $f \in C^{\infty}(m)$. Then for all $p \in M$ we have
\begin{equation*}\label{lemmaequality}
    \Delta_\phi f = \overline{\Delta_\phi}f + (n-1)H_\phi \frac{\partial f}{\partial \nu} + \nabla^2 f(\nu, \nu),
\end{equation*}
where $H_\phi$ is the weighted curvature of $\partial M$, $\Delta_\phi$ and $\overline{\Delta}_\phi$ are the drifted Laplacian operators defined in $M$ and $\partial M$, respectively.
\end{lem}
The next lemma is a slight modification of Theorem~1.6 in \cite{HuangRuan} and Proposition~2.2 in \cite{BatistaSantos}.
\begin{lem}\label{lemma2.1}
Let $M$ an  $n$-dimensional metric measure space with nonempty boundary $\partial M$ and $\rc_\phi^m \geq 0$. If the second fundamental form of $\partial M$ satisfies $II \geq cI$ (in the quadratic form sense) and $H_\phi \geq c$  for some positive constant $c$, then 
\begin{equation*}
    \eta_1 \geq (n-1)c^2,
\end{equation*}
where $\eta_1$ is the first nonzero eigenvalue of the drifted Laplacian acting on functions on $\partial M$. The equality holds if and only if $M$ is isometric to an Euclidean ball of radius $\frac{1}{c}$, $f$ is constant and $m=n$.
\end{lem}

Now, we are in a position to give the proof of the theorems of this paper.
 
\begin{proof}[Proof of Theorem~\ref{clampledteo}]
Let $u$ be the first eigenfunction corresponding to the first eigenvalue $\Gamma_1$, that is, 
\begin{equation}\label{clampledplate2}
    \left \{ \begin{matrix}\Delta_\phi^2 u = \Gamma_1 u & \mbox{in} \quad M,  \\
    u =\frac{\partial u}{\partial \nu} = 0 & \mbox{on} \quad \partial M. \end{matrix} \right.
\end{equation}
Since $u=\frac{\partial u}{\partial \nu}=0$ on $\partial M$, substituting $u$ into Reilly's formula \eqref{reillyequality}, we get \begin{equation}\label{reilly53}
    \int_M(\Delta_\phi u)^2 d\mu=\int_M|\nabla^2u|^2d\mu + \int_M \rc_\phi(\nabla u, \nabla u)d\mu .
\end{equation}
By the definition of drifted Laplacian and from Schwarz inequality, using Inequality~\eqref{algebric-ineq}, we have 
\begin{equation}\label{shwarz54}
    |\nabla^2u|^2 \geq \frac{1}{n}\Delta u = \frac{1}{n}(\Delta_\phi u + \langle \nabla \phi, \nabla u \rangle)^2 \geq \frac{(\Delta_\phi u)^2}{n(a+1)} - \frac{|\nabla \phi|^2|\nabla u|^2}{na}.
\end{equation}
with equality holding if and only if $|\nabla^2u|^2 = \frac{1}{n}\Delta u$ and $\Delta u +\frac{1}{a}\langle \nabla \phi, \nabla u \rangle=0$.
From \eqref{clampledplate2} and divergence theorem, we obtain 
\begin{equation}\label{equation55}
    \int_M (\Delta_\phi u)^2d\mu = \int_M u\Delta_\phi^2ud\mu = \Gamma_1\int_M u^2 d\mu.
\end{equation}
Since $\rc_\phi \geq \frac{|\nabla \phi|^2}{na}+b$, from \eqref{reilly53} and \eqref{shwarz54}, we have 
\begin{equation}\label{eq56}
    \int_M(\Delta_\phi u)^2d\mu \geq \frac{1}{n(a+1)}\int_M(\Delta_\phi u)^2d\mu + b\int_M |\nabla u|^2d\mu.
\end{equation}
Moreover, since $u$ is not a zero function which vanishes on $\partial M$, we know that
\begin{equation}\label{eq57}
    \int_M(\Delta_\phi u)^2d\mu \geq \lambda_1 \int_M |\nabla u|^2d\mu \geq \lambda_1^2 \int_M u^2 d\mu,
\end{equation}
with equality holding if and only if $u$ is a first eigenfunction of the Dirichlet problem for drifted Laplacian of $M$. Thus, by \eqref{equation55}, \eqref{eq56} and \eqref{eq57} we conclude
\begin{equation*}
    \Gamma_1 \geq \lambda_1\Big(\frac{\lambda_1}{n(a+1)}+b\Big).
\end{equation*}
Suppose that $ \Gamma_1 = \lambda_1(\frac{\lambda_1}{n(a+1)}+b)$ is valid, then equality holds in \eqref{shwarz54} and we get that $\phi$ is not a constant and 
\begin{equation*}
    \Delta u + \frac{1}{a}\langle \nabla \phi , \nabla u \rangle=0,
\end{equation*}
 everywhere on $M$. Multiplying the previous inequality by $u$ and integrating on $M$ with respect to $e^{\frac{1}{a}\phi}dv$ we get
\begin{equation*}
    0=\int_M u(\Delta u + \frac{1}{a}\langle \nabla \phi , \nabla u \rangle)e^{\frac{1}{a}\phi}dv=-\int_M |\nabla u|^2 e^{\frac{1}{a}\phi}dv,
\end{equation*}
then $u$ is a constant function on $M$, which is a contradiction since $u$ is the first eigenfunction of Bi-drifted Laplacian and cannot be a constant. Therefore, we complete the proof of Item~\eqref{thm1-1}.

Now, let is use Inequality~\eqref{reillyinequality} to proof Item~\eqref{thm1-2}. In fact, substituting $u$ into \eqref{reillyinequality}, we have
\begin{align*}
    \int_M(\Delta_\phi u)^2d\mu &\geq \frac{1}{m}\int_M(\Delta_\phi u)^2d\mu + \int_M \rc_\phi^m(\nabla u, \nabla u)d\mu \nonumber \\
    &\geq \frac{1}{m}\int_M(\Delta_\phi u)^2d\mu + (m-1)k\int_M|\nabla u|^2d\mu.
\end{align*}
Thus, by \eqref{equation55}, \eqref{eq57}  and the previous inequality  we conclude that 
\begin{equation*}
    \Gamma_1 \geq \lambda_1 \Big(\frac{\lambda_1}{m}+(m-1)k\Big). 
\end{equation*}
When $m=n$, then $\phi$ is constant and $\Delta_\phi = \Delta$, by arguments similar to the final part of the proof of \cite[Theorem~1.1]{wangxia4}, we have $\Gamma_1 > \lambda_1(\frac{\lambda_1}{m}+(m-1)k) $.\\
When $m>n$, if $\Gamma_1 = \lambda_1(\frac{\lambda_1}{m}+(m-1)k)$, we know that the equality holds in \eqref{reillyinequality}, which means by Remark~\ref{remar2.1}
\begin{equation*}
    0=\Delta_\phi u + \frac{m}{m-n}\langle \nabla \phi, \nabla u \rangle = \Delta u + \frac{n}{m-n}\langle \nabla \phi, \nabla u \rangle,
\end{equation*}
holds everywhere on $M$. Multiplying the previous equation by $u$ and integrating on $M$ with respect to $e^{\frac{n}{m-n}\phi}dv$ we have
\begin{equation*}
    0=\int_M u(\Delta u + \frac{n}{m-n}\langle \nabla \phi , \nabla u \rangle)e^{\frac{n}{m-n}\phi}dv=-\int_M |\nabla u|^2 e^{\frac{n}{m-n}\phi}dv.
\end{equation*}
From previous equality, we know that $u$ is a constant function on $M$, which is a contradiction since $u$ is the first eigenfunction of drifted Laplacian and cannot be a constant. Therefore, we finish the proof of Item~\eqref{thm1-2}. Thus we complete the proof of Theorem~\ref{clampledteo}.
\end{proof}

\begin{proof}[Proof of Theorem~\ref{bucklingteo}]
 Let $g$ be the eigenfunction corresponding to the first eigenvalue $\Lambda_1$ of Problem~\eqref{buckling}, that is,
\begin{equation}\label{buckling2}
    \left \{ \begin{matrix}\Delta_\phi^2 g = -\Lambda_1 \Delta_\phi g & \mbox{in}\quad M,  \\
    g =\frac{\partial g}{\partial \nu} = 0 & \mbox{on} \quad \partial M. \end{matrix} \right.
\end{equation}
Using divergence theorem and the first equality in \eqref{buckling2} we have
\begin{equation}\label{Lambda1}
    \int_M(\Delta_\phi g)^2 d\mu = - \int_M g \Delta_\phi g d\mu = \Lambda_1 \int_M |\nabla g|^2 d\mu. 
\end{equation}
From Reilly's formula \eqref{reillyequality} and $\rc_\phi\geq \frac{|\nabla \phi|^2}{na}+b$, we obtain
\begin{align}\label{reilly65}
    \int_M(\Delta_\phi g)^2d\mu&=\int_M|\nabla^2g|^2d\mu + \int_M \rc_\phi(\nabla g, \nabla g)d\mu \nonumber\\
    &\geq \int_M |\nabla^2g|^2 d\mu + \int_M (\frac{|\nabla \phi|^2}{na}+b)|\nabla g|^2 d\mu.
\end{align}
Similar to \eqref{shwarz54} we have
\begin{equation*}
    |\nabla^2g|^2 \geq \frac{(\Delta_\phi g)^2}{n(a+1)} - \frac{|\nabla \phi|^2|\nabla g|^2}{na}.
\end{equation*}
with equality holding if and only if $|\nabla^2g|^2 = \frac{1}{n}\Delta g$ and $\Delta g +\frac{1}{a}\langle \nabla \phi, \nabla g \rangle=0$.
From \eqref{reilly65} and the previous inequality we get
\begin{equation}\label{eqqq56}
    \int_M(\Delta_\phi g)^2d\mu \geq \frac{1}{n(a+1)}\int_M (\Delta_\phi g)^2 d\mu + b\int_M |\nabla g|^2 d\mu.
\end{equation}
On the other hand, since $g$ is not a zero function which vanishes on $\partial M$, we know that
\begin{equation}\label{eq68}
    \int_M(\Delta_\phi g)^2d\mu \geq \lambda_1 \int_M |\nabla g|^2d\mu,
\end{equation}
with equality holding if and only if $g$ is a first eigenfunction of the Dirichlet problem for drifted Laplacian of $M$. 
Thus, from \eqref{Lambda1}, \eqref{eqqq56} and \eqref{eq68} we conclude that
\begin{equation*}\label{eq69}
    \Lambda_1 \geq \frac{\lambda_1}{n(a+1)} + b.
\end{equation*}
Analogous to what was done in the proof of Theorem~\ref{clampledteo}, if we suppose that $\Lambda_1 = \frac{\lambda_1}{n(a+1)} + b$, then
\begin{equation*}
    0=\int_M g(\Delta g + \frac{1}{a}\langle \nabla \phi , \nabla g \rangle)e^{\frac{1}{a}\phi}dv=-\int_M |\nabla g|^2 e^{\frac{1}{a}\phi}dv.
\end{equation*}
which is a contradiction since $g$ is the first eigenfunction of \eqref{buckling2} and cannot be a constant. Therefore, we have $\Lambda_1 > \frac{\lambda_1}{n(a+1)} + b$ and complete the proof of Theorem~\ref{bucklingteo}.
\end{proof}

\begin{proof}[Proof of Theorem~\ref{zeta1estimateteo}]
As a consequence of Inequality~\eqref{eq-zeta1-xi1} we must show only that
\begin{equation*}
    \xi_1 > \frac{mc\eta_1 \mu_1}{(m-1)(\mu_1+mk)},
\end{equation*}
where $\xi_1$ is the first nonzero eigenvalue of Problem~\eqref{problemxi1}.
For this, let $f$ be the eigenfunction corresponding $\xi_1$ of Problem~\eqref{problemxi1}, that is, 
\begin{align}\label{eq61}
    \left\{ \begin{matrix}\Delta_\phi^2 f = 0 & \mbox{in} \quad M,  \\
    \frac{\partial f}{\partial \nu} = \frac{\partial (\Delta_\phi f)}{\partial \nu} + \xi_1 f = 0 & \mbox{on} \quad \partial M. \end{matrix} \right.
\end{align}
Let $z=f|_{\partial M}$, then $z\neq 0$ and from \eqref{eq61} we get 
\begin{equation}\label{eq62}
    \xi_1 = \frac{\int_M (\Delta_\phi f)^2d\mu}{\int_{\partial M}z^2 d\vartheta}.
\end{equation}
Using that $\rc_\phi^m(\nabla f, \nabla f) \geq -(m-1)k|\nabla f|^2$ and substituting $f$ into \eqref{reillyinequality}, we have 
\begin{align}\label{eq63}
    \frac{m-1}{m}\int_M(\Delta_\phi f)^2d\mu  &\geq  \int_M \rc_\phi^m (\nabla f, \nabla f)d\mu + \int_{\partial M}II(\overline{\nabla}z, \overline{\nabla}z) d\vartheta\nonumber \\
    &\geq -(m-1)k\int_M |\nabla f|^2d\mu + c\int_{\partial M} |\overline{\nabla}z|^2 d\vartheta.
\end{align}
Since $\frac{\partial f}{\partial \nu}|_{\partial M}=0$, we know that
\begin{equation}\label{equation2.24}
    \int_M(\Delta_\phi f)^2d\mu \geq \mu_1 \int_M |\nabla f|^2d\mu.
\end{equation}
From \eqref{eq61} we can see that $\int_{\partial M}z d\vartheta=0$. Indeed, from \eqref{eq61} we get
\begin{align*}
    \xi_1 \int_{\partial M}fe^{-\phi}dA &= \int_{\partial M} \frac{\partial}{\partial \nu}(f-\overline{\Delta}_\phi f)e^{-\phi}dA \nonumber\\
    &=\int_M \dv((\nabla f - \nabla \Delta_\phi f)e^{-\phi})dv \nonumber \\
    &=\int_M(\Delta_\phi f -  \Delta_\phi^2 f)e^{-\phi}dv\nonumber\\
    &=\int_M(\Delta_\phi f)e^{-\phi}dv = 0.
\end{align*}
So, we have from Poincaré's inequality that 
\begin{equation}\label{eq66}
    \int_{\partial M}|\overline{\nabla}z|^2 d\vartheta \geq \eta_1 \int_{\partial M}z^2d\vartheta. 
\end{equation}
From \eqref{eq63}-\eqref{eq66}, we get
\begin{align*}
    \frac{m-1}{m}\int_M(\Delta_\phi f)^2d\mu \geq -\frac{(m-1)k}{\mu_1}\int_M (\Delta_\phi f)^2d\mu + c\eta_1\int_{\partial M} z^2 d\vartheta, 
\end{align*}
therefore, from \eqref{eq62} and the previous inequality, we obtain 
\begin{equation*}
   \xi_1 \geq \frac{mc\eta_1\mu_1}{(m-1)(mk+\mu_1)} .
\end{equation*}
The equality in the previous inequality can not occur. In fact, otherwise, we must have a sign of equality in \eqref{reillyinequality} and consequently (see Remark~\ref{remar2.1})
\begin{equation*}
    \nabla^2f=\frac{\Delta f}{n}\langle, \rangle \quad \mbox{and} \quad \Delta_\phi f + \frac{m}{m-n}\langle \nabla \phi, \nabla f \rangle=0.
\end{equation*}
Thus, for a tangent vector field $X$ of $\partial M$, from the first equation above and of the fact $\frac{\partial f}{\partial \nu}|_{\partial M}=0$ we obtain 
\begin{equation*}
    0=\nabla^2f(\nu, X) = X\nu(f) - (\nabla_X\nu)f=-\langle \nabla_X \nu,\overline{\nabla} z\rangle.
\end{equation*}
In particular, taking $X=\overline{\nabla}z$ we have
\begin{equation*}
    II(\overline{\nabla}z, \overline{\nabla}z)=\langle \nabla_{\overline{\nabla}z}\nu, \overline{\nabla}z \rangle = 0.
\end{equation*}
This is impossible since $II=cI$ and $z$ is not constant. This completes the proof of Theorem~\ref{zeta1estimateteo}.
\end{proof}

\begin{proof}[Proof of Theorem~\ref{teorematau1}]
Let $f$ be an eigenfunction corresponding to $\tau_1$, that is:
\begin{equation}\label{bi-drifting2}
    \left \{ \begin{matrix}\Delta_\phi^2 f = 0 & \mbox{in} \quad M,  \\
    \frac{\partial f}{\partial \nu} = \frac{\partial (\Delta_\phi f)}{\partial \nu} - \tau_1 \overline{\Delta}_\phi f = 0 & \mbox{on} \quad \partial M. \end{matrix} \right.
\end{equation}
Set $z=f|_{\partial M}$, then $z\neq 0$ and $h=\partial_{\nu} f|_{\partial M}=0$. From \eqref{bi-drifting2}, we get
\begin{equation}\label{eq-tau1}
    \tau_1 =  \frac{\int_M (\Delta_\phi f)^2 d\mu}{\int_{\partial M}|\nabla z|^2 d\vartheta}. 
\end{equation}
Substituting $f$ into Reilly's formula \eqref{reillyequality}, we have
\begin{align}\label{eq26}
    \int_{M}\Big((\Delta_\phi f)^2 - |\nabla^2f|^2 \Big)d\mu &= \int_{M}\rc_{\phi}(\nabla f,\nabla f)d\mu + \int_{\partial M}II(\overline{\nabla}z, \overline{\nabla} z)d\vartheta  \nonumber\\
    &\geq \int_M \Big( \frac{|\nabla \phi|^2}{na} - b\Big) |\nabla f|^2 d\mu + c\int_{\partial M}|\overline{\nabla} z|^2d\vartheta.
\end{align}
Analogous to \eqref{shwarz54}, we can see that
\begin{equation}\label{eq29}
    |\nabla^2f|^2 \geq \frac{(\Delta_\phi f)^2}{n(a+1)} - \frac{|\nabla \phi|^2|\nabla f|^2}{na}.
\end{equation}
with equality holding if and only if $|\nabla^2f|^2 = \frac{1}{n}\Delta f$ and $\Delta f +\frac{1}{a}\langle \nabla \phi, \nabla f \rangle=0$.
From \eqref{eq26} and \eqref{eq29}, we get
\begin{equation}\label{eq30}
    \Big(1- \frac{1}{n(a+1)}\Big)\int_M(\Delta_\phi f)^2d\mu \geq -b\int_M |\nabla f|^2d\mu + c\int_{\partial M} |\overline{\nabla}z|^2ed\vartheta.
\end{equation}
Since $\frac{\partial f}{\partial \nu}|_{\partial M}=0$, we have 
\begin{equation}\label{neumannmu1}
    \int_M (\Delta_\phi f)^2d\mu \geq \mu_1 \int_M |\nabla f|^2 d\mu.
\end{equation}
It follows from \eqref{eq30} and \eqref{neumannmu1} that
\begin{equation}\label{eq97}
 \Big(1- \frac{1}{n(a+1)}+ \frac{b}{\mu_1}\Big)\int_M(\Delta_\phi f)^2d\mu \geq c\int_{\partial M} |\overline{\nabla} z|^2d\vartheta.
\end{equation}
Thus, from \eqref{eq-tau1} and \eqref{eq97} we conclude 
\begin{equation}\label{zeta2}
    \tau_1 \geq \frac{nc(a+1)\mu_1}{n(a+1)(\mu_1+b)-\mu_1}.
\end{equation}
Let us show by contradiction that the equality in \eqref{zeta2} can not occur. In fact, if \eqref{zeta2} take equality sign, then we must have  equality sign in \eqref{eq29} that implies
\begin{equation*}
    \nabla^2f = \frac{\Delta f}{n} \langle , \rangle.
\end{equation*}
Thus for a tangent vector field $X$ of $\partial M$, we have from the previous equality and $\frac{\partial f}{\partial \nu}|_{\partial M}=0$ that 
\begin{equation*}
    0=\nabla^2f(\nu, X) = X\nu(f) - (\nabla_X \nu)(f) = - \langle \nabla_X \nu, \overline{\nabla} z \rangle. 
\end{equation*}
In particular, for $X=\overline{\nabla}z$, we have $II(\overline{\nabla}z, \overline{\nabla}z)=0$. 
This is impossible since $II=cI$ and $z$ is not constant. This finish the proof of Item~\eqref{thm4-1}.

Now, we can use Inequality~\eqref{reillyinequality} to proof Item~\eqref{thm4-2}. In fact, substituting $f$ into \eqref{reillyinequality} and since $\rc_\phi^m \geq -(m-1)k$, we have
\begin{align}\label{eqq44}
    \frac{m-1}{m}\int_M(\Delta_\phi f)^2d\mu &\geq \int_M \rc_\phi^m(\nabla f, \nabla f)d\mu + \int_{\partial M}II(\overline{\nabla}z,\overline{\nabla}z)d\vartheta\nonumber\\
    &\geq-(m-1)k\int_M|\nabla f|^2d\mu +c\int_{\partial M}|\overline{\nabla}z|^2d\vartheta.
\end{align}
From \eqref{neumannmu1} and \eqref{eqq44}  we obtain
\begin{equation*}
    \Big( \frac{m-1}{m} + \frac{(m-1)k}{\mu_1} \Big)\int_M (\Delta_\phi f)^2d\mu \geq c \int_{\partial M}|\overline{\nabla}z|^2d\vartheta.
\end{equation*}
Then, from \eqref{eq-tau1} and the previous inequality, we have
\begin{equation*}
    \tau_1 \geq \frac{mc\mu_1}{(m-1)(mk+\mu_1)}.
\end{equation*}
For an argument similar to the end of proof of Theorem~\ref{zeta1estimateteo} and to the end of proof of Item~\eqref{thm4-1} we obtain $\tau_1 > \frac{mc\mu_1}{(m-1)(mk+\mu_1)}$ and we conclude the proof of Item~\eqref{thm4-2}. This completes the proof of Theorem~\ref{teorematau1}.
\end{proof}

\begin{proof}[Proof of Theorem~\ref{volumproblemteo}] Let $f$ be the solution of the following Laplace equation
\begin{equation*}
    \left \{ \begin{matrix}\Delta_\phi f = 1 & \mbox{in} \quad M,  \\
    f = 0 & \mbox{on} \quad \partial M. \end{matrix} \right.
\end{equation*}
It follows from Rayleigh-Ritz characterization of $p_1$ that 
\begin{equation}\label{eq2.33}
    p_1 \leq \frac{\int_M (\Delta_\phi f)^2e^{-\phi}dv}{\int_{\partial M} \eta^2e^{-\phi}dA}= \frac{\int_M e^{-\phi}dv}{\int_{\partial M} \eta^2e^{-\phi}dA} = \frac{V_\phi}{\int_{\partial M}\eta^2e^{-\phi}dA},
\end{equation}
where $\eta = \frac{\partial f}{\partial \nu}|_{\partial M}$. Integrating $\Delta_\phi f=1$ on $M$ and using the divergence theorem it gives 
\begin{equation*}
    V_\phi=\int_M\Delta_\phi f e^{-\phi}dv = \int_M \dv(e^{-\phi}\nabla f)dv = \int_{\partial M} \eta e^{-\phi}dA.
\end{equation*}
Hence we infer from Schwarz's inequality that 
\begin{equation*}\label{eq105}
    V_{\phi}^2=\Big(\int_{\partial M} \eta e^{-\phi}dA\Big)^2\leq \int_{\partial M}(\eta e^{-\frac{\phi}{2}})^2dA \int_{\partial M}(e^{-\frac{\phi}{2}})^2dA,
\end{equation*}
hence, we obtain
\begin{equation}\label{eqq105}
     V_{\phi}^2\leq A_\phi\int_{\partial M} \eta^2 e^{-\phi}dA.
\end{equation}
Thus, from \eqref{eq2.33} and \eqref{eqq105}, we have 
\begin{equation*}
    p_1 \leq \frac{V_\phi}{\int_{\partial M}e^{-\phi}dA} \leq \frac{A_\phi}{V_\phi}.
\end{equation*}
Now, assume that $\rc_\phi \geq \frac{|\nabla \phi|^2}{na}$, $H_\phi \geq \frac{n(a+1)-1}{n(a+1)(n-1)}\frac{A_\phi}{V_\phi}$ for some $x_0 \in \partial M$ and $p_1=\frac{A_\phi}{V_\phi}$. In this case, the equality must hold in \eqref{eqq105} and so $\eta = \frac{V_\phi}{A_\phi}$ is a constant. Consider the function $\psi$ on $M$ given by 
\begin{equation}\label{eq106}
    \psi = \frac{1}{2}|\nabla f|^2 - \frac{f}{n(a+1)}.
\end{equation}
Using the Bochner formula for the drifted Laplacian (see \cite{SchoenYau}), that is, 
\begin{equation*}
    \frac{1}{2}\Delta_\phi (|\nabla f|^2) = |\nabla ^2 f|^2 + \langle \nabla f, \nabla (\Delta_\phi f) \rangle + \rc_\phi(\nabla f, \nabla f),
\end{equation*}
and \eqref{eq106}, since $\rc_\phi(\nabla f, \nabla f) \geq \frac{|\nabla \phi|^2|\nabla f|^2}{na}$, we have 
\begin{align}\label{eq108}
    \Delta_\phi \psi &= \frac{1}{2}\Delta_\phi(|\nabla f|^2) - \frac{\Delta_\phi f}{n(a+1)}\nonumber\\
            &\geq    |\nabla ^2 f|^2 + \langle \nabla f, \nabla (\Delta_\phi f) \rangle + \frac{|\nabla \phi|^2|\nabla f|^2}{na} - \frac{1}{n(a+1)}.
\end{align}
From Schwarz's inequality and using Inequality~\eqref{algebric-ineq}, we get 
\begin{equation}\label{eq109}
|\nabla^2 f|^2 \geq \frac{1}{n}(\Delta f)^2 \geq \frac{(\Delta_\phi f)^2}{n(a+1)} - \frac{|\nabla \phi|^2 |\nabla f|^2}{na}.    
\end{equation}
From \eqref{eq108} and \eqref{eq109} we have $\Delta_\phi \psi \geq 0$. Moreover, since $\eta=\frac{V_\phi}{A_\phi}$, we can see that $\psi = \frac{1}{2}(\frac{V_\phi}{A_\phi})^2$ on the boundary. So, we conclude from the strong maximum principle and Hopf Lemma (see \cite{GilbargTrudinger}, pp. 34-35) that either 
\begin{equation*}
    \psi = \frac{1}{2}\Big(\frac{V_\phi}{A_\phi}\Big)^2 \quad \mbox{in} \quad M,
\end{equation*}
or
\begin{equation}\label{eq111}
    \frac{\partial \psi}{\partial \nu}(y) > 0, \quad \forall y \in \partial M.
\end{equation}
Since $f|_{\partial M}=0$ and $\Delta_\phi f=1$, from Lemma~\ref{lemma} we obtain
\begin{equation*}
    1=(\Delta_\phi f)|_{\partial M} = (n-1)H_\phi \eta + \nabla^2f(\nu, \nu).
\end{equation*}
Hence, it follows on $\partial M$ that
\begin{align*}
    \frac{\partial \psi}{\partial \nu} &= \eta \nabla^2 f(\nu, \nu) - \frac{\eta}{n(a+1)}\nonumber \\
    &=\eta(1-(n-1)H_\phi \eta) - \frac{\eta}{n(a+1)}\nonumber\\
    &=\big( \frac{n(a+1)-1}{n(a+1)} - (n-1)H_\phi \frac{V_\phi}{A_\phi} \big)\frac{V_\phi}{A_\phi},
\end{align*}
which shows that \eqref{eq111} is not true since $H_\phi(x_0)\geq \frac{n(a+1)-1}{n(n-1)(a+1)}\frac{A_\phi}{V_\phi}$. Therefore $\psi$ is constant on $M$. Since $\Delta_\phi \psi =0$, that is, the equations \eqref{eq108} and \eqref{eq109} hold and so
\begin{equation*}
\Delta f
+\frac{1}{a}\langle \nabla \phi, \nabla f \rangle = 0    .
\end{equation*}
Multiplying the equality above with $f$ and integrating on $M$ with respect to $e^{\frac{1}{a}\phi}dv$ we get
\begin{equation*}
    0=\int_Mf(\Delta f +\frac{1}{a}\langle \nabla \phi, \nabla f \rangle)e^{\frac{1}{a}\phi}dv = -\int_M |\nabla u|^2e^{\frac{1}{a}\phi}dv.
\end{equation*}
Therefore, we have that $f$ is a constant function on $M$, which is a contradiction since $\Delta_\phi f =1$. Thus, we conclude that the inequalities in \eqref{eq108} and \eqref{eq109} hold only when $\phi=constant$ and $\rc_\phi = \rc$. Then by Wang and Xia's argument in \cite[Theorem~1.3]{wangxia1}, we complete the proof of Theorem~\ref{volumproblemteo}.
\end{proof}

\begin{proof}[Proof of Theorem~\ref{wentzell-theorem}]
Let us consider $u$ be the solution to the following problem
\begin{equation*}
    \left \{ \begin{array}{ll}\Delta_\phi u = 0 & \mbox{in } M,  \\
    u = z & \mbox{on } \partial M, \end{array} \right.
\end{equation*}
where $z$ is the eigenfunction of the first nonzero closed eigenvalue $\eta_1$ of the drifted Laplacian on $\partial M$, that is, $\overline{\Delta}_\phi z + \eta_1 z=0$ on $\partial M$. Set $h=\frac{\partial u}{\partial \nu}$, from \eqref{gamma1} and the fact $\int_{\partial M}ud\mu = \int_{\partial M}zd\mu=0$, we have
\begin{align}\label{equation-2.43}
\gamma_{1, \phi} \leq \frac{\int_M|\nabla u|^2d\vartheta+\rho\int_{\partial M}|\overline{\nabla}z|^2 d\mu}{\int_{\partial M} z^2 d\mu} &= \frac{\int_{\partial M} zh d\mu+\rho\int_{\partial M}|\overline{\nabla}z|^2 d\mu}{\int_{\partial M} z^2 d\mu}\nonumber\\
&=\rho \eta_1 + \frac{\int_{\partial M} zh d\mu}{\int_{\partial M} u^2 d\mu}.
\end{align}
By hypothesis, $\rc_\phi^m(\nabla u, \nabla u)\geq -k|\nabla u|^2$, $II(\overline{\nabla}z, \overline{\nabla}z)\geq c |\overline{\nabla}z|^2$, $H_\phi\geq c$ and using Inequality \eqref{reillyinequality}, we get 
\begin{align*}
    k\int_{\partial M}zhd\mu &= k\int_M|\nabla u|^2d\vartheta \geq \int_M [\frac{m-1}{m}(\nabla_\phi u)^2-\rc_\phi^m(\nabla u, \nabla u)]d\vartheta\\
    &\geq \int_{\partial M}\Big[(n-1)H_\phi h^2 + 2(\overline{\Delta}_\phi z)h+II(\overline{\nabla}z, \overline{\nabla}z)\Big]d\mu\\
    &\geq (n-1)c\int_{\partial M}h^2d\mu-2\eta_1\int_{\partial M}zhd\mu+c\int_{\partial M}|\overline{\nabla}z|^2d\mu,
\end{align*}
then,
\begin{equation}\label{equation-2.44}
    0\geq (n-1)c\int_{\partial M}h^2d\mu-(2\eta_1+k)\int_{\partial M}zhd\mu+c\eta_1\int_{\partial M}z^2d\mu .
\end{equation}
Thus, we have
\begin{align*}
    0&\geq (n-1)c\int_{\partial M}\Big(h-\frac{2\eta_1+k}{2(n-1)c}z\Big)^2d\mu + \Big(c\eta_1 - \frac{(2\eta_1+k)^2}{4(n-1)c}\Big)\int_{\partial M}z^2d\mu\\
    &\geq \Big(c\eta_1 - \frac{(2\eta_1+k)^2}{4(n-1)c}\Big)\int_{\partial M}z^2d\mu,
\end{align*}
and so 
\begin{equation*}
    (2\eta_1+k)^2\geq 4(n-1)\eta_1c^2.
\end{equation*}
We can also get from \eqref{equation-2.44} 
\begin{equation*}
    0\geq (n-1)c\int_{\partial M}h^2d\mu-(2\eta_1+k)\Big(\int_{\partial M}h^2d\mu\Big)^{\frac{1}{2}}\Big(\int_{\partial M}z^2d\mu\Big)^{\frac{1}{2}}+c\eta_1\int_{\partial M}z^2d\mu ,
\end{equation*}
which, implies that
\begin{equation*}
    \Big(\int_{\partial M}h^2d\mu\Big)^{\frac{1}{2}} \leq \frac{2\eta_1+k+\sqrt{(2\eta_1+k)^2-4(n-1)\eta_1c^2}}{2(n-1)c}\Big(\int_{\partial M}z^2d\mu\Big)^{\frac{1}{2}}.
\end{equation*}
Combining \eqref{equation-2.43} and the previous inequality, we obtain \eqref{gamma1estimate}. If equality holds in \eqref{gamma1estimate}, then all inequalities become equalities, and through the above argument, we have $\nabla^2 u=0, H_\phi=c$ and
\begin{equation*}
    h=\frac{2\eta_1+k+\sqrt{(2\eta_1+k)^2-4(n-1)\eta_1c^2}}{2(n-1)c}z.
\end{equation*}
Taking a local orthonormal fields $\{e_i\}_{i=1}^{n-1}$ tangent to $\partial M$. We infer using Lemma~\ref{lemma} and the previous equalities that
\begin{align*}
    0=&\sum_{i=1}^{n-1}\nabla^2u(e_i, e_i)= \overline{\Delta}_\phi z+(n-1)H_\phi h \nonumber \\
    =&-\eta_1 z + (n-1)c\cdot \frac{2\eta_1+k+\sqrt{(2\eta_1+k)^2-4(n-1)\eta_1c^2}}{2(n-1)c}z,
\end{align*}
which implies $k=0$ and $\eta_1=(n-1)c^2$. Then, under the assumptions for $\rc_\phi^m$, $II$ and $H_\phi$ by Lemma~\ref{lemma2.1} we know that $M$ is isometric to an Euclidean ball of radius $\frac{1}{c}$, $\phi$ is constant and $m=n$. On the other hand, we know that for the $n$-dimensional Euclidean ball of radius $\frac{1}{c}$, the equality in \eqref{gamma1estimate} holds when $\phi$ is constant and $m=n$(see \cite[Eq. (1.7)]{wangxia3}). This completes the proof of Theorem~\ref{wentzell-theorem}.
\end{proof}

\begin{proof}[Proof of Theorem~\ref{steklov-theorem}]
Let $u$ be an eigenfunction corresponding to the first eigenvalue $q_1$ of Problem~\eqref{steklov-problem}, that is,
\begin{equation*}
    \left \{ \begin{array}{ll}\Delta_\phi u = 0 & \mbox{in } M,  \\
    \frac{\partial u}{\partial \nu} = q_1 u & \mbox{on } \partial M. \end{array} \right.
\end{equation*}
Setting $w=u|_{\partial M}$, $y=\frac{\partial u}{\partial \nu}|_{\partial M}$.  Since 
\begin{equation*}
    \rc_\phi^m(\nabla u, \nabla u)\geq -k|\nabla u|^2, \quad II(\overline{\nabla}w,\overline{\nabla}w)\geq c|\overline{\nabla}w|^2, \quad \mbox{and}\quad H_\phi\geq c,
\end{equation*}
substituting $u$ into Inequality~\eqref{reillyinequality}, we obtain
\begin{align*}
    k\int_M |\nabla u|^2 d\vartheta &\geq \int_{\partial M}[(n-1)H_\phi y^2+2(\overline{\Delta}_\phi w)y + II(\overline{\nabla}w, \overline{\nabla}w)]d\mu\nonumber\\
    &\geq (n-1)c\int_{\partial M}y^2 d\mu + 2 \int_{\partial M}(\overline{\Delta}_\phi w)yd\mu + c\int_{\partial M}|\overline{\nabla} w|^2 d\mu . 
\end{align*}
By divergence theorem $\int_{\partial M}y \overline{\Delta}_\phi w d\mu = - \int_{\partial M} \langle \overline{\nabla}y, \overline{\nabla}w \rangle d\mu$, then
\begin{align}\label{equation2.44}
    k\int_M |\nabla u|^2 d\vartheta > - 2\int_{\partial M} \langle \overline{\nabla}y, \overline{\nabla}w \rangle d\mu + c\int_{\partial M}|\overline{\nabla} w|^2 d\mu=(-2q_1+c)\int_{\partial M}|\overline{\nabla} w|^2 d\mu . 
\end{align}
Since $\int_{\partial M} w d\mu = 0$ and $w\neq0$, we know that
\begin{equation*}
    \int_{\partial M} w^2d\mu \leq \frac{1}{\eta_1}\int_{\partial M}|\overline{\nabla} w|^2 d\mu.
\end{equation*}
Therefore, from divergence theorem and the previous inequality we get
\begin{equation*}
    \int_M |\nabla u|^2d\vartheta = \int_M \dv (u\nabla u) d\vartheta = \int_{\partial M}wyd\mu = q_1\int_{\partial M}w^2 d\mu \leq \frac{q_1}{\eta_1}\int_{\partial M}|\overline{\nabla} w|^2 d\mu .
\end{equation*}
From the previous inequality and Inequality~\eqref{equation2.44} we obtain 
\begin{equation*}
   q_1 > \frac{c\eta_1}{2\eta_1+k}, 
\end{equation*}
and completes the proof of Theorem\eqref{steklov-theorem}.
\end{proof}

\section*{Acknowledgements} 
The author would like to express their sincere thanks to Xia Changyu and the referee for useful comments, discussions and constant encouragement. This study was financed in part by Coordenação de Aperfeiçoamento de Pessoal de Nível Superior (CAPES) in conjunction with Fundação Rondônia de Amparo ao Desenvolvimento~das Ações Científicas e Tecnológicas e à Pesquisa do Estado de Rondônia (FAPERO) - Finance Code 001.

\end{document}